\documentclass[11pt]{article}
\usepackage[a4paper,total={6in,9in}]{geometry}
\usepackage{amsmath,amsfonts,amssymb}
\usepackage{amsthm}
\usepackage{bbm}
\usepackage{verbatim}
\usepackage[hidelinks,pdfstartview=FitH]{hyperref}

\theoremstyle{plain}
\newtheorem{theorem}{Theorem}[section]
\newtheorem{lemma}[theorem]{Lemma}
\newtheorem{corollary}[theorem]{Corollary}
\newtheorem{proposition}[theorem]{Proposition}

\theoremstyle{remark}
\newtheorem{definition}[theorem]{Definition}
\newtheorem{remark}[theorem]{Remark}

\begin{document}

\title{Transportation cost inequalities for stochastic reaction diffusion equations on the whole line $\mathbb{R}$}

\author{Yue Li$^{1}$, Shijie Shang$^{2}$ and Tusheng Zhang$^{3}$}
\footnotetext[1]{\, School of Mathematics, University of Science and technology of China, Hefei, China. Email: liyue27@mail.ustc.edu.cn}
\footnotetext[2]{\, School of Mathematics, University of Science and technology of China, Hefei, China. Email: sjshang@ustc.edu.cn}
\footnotetext[3]{\, Department of Mathematics, University of Manchester, Oxford Road, Manchester M13 9PL, England, U.K. Email: tusheng.zhang@manchester.ac.uk}
\maketitle

\begin{abstract}
In this paper, we established quadratic transportation cost inequalities for solutions of stochastic reaction diffusion equations driven by multiplicative space-time white noise on the whole line $\mathbb{R}$. Since the space variable is defined on the unbounded domain $\mathbb{R}$, the inequalities are proved under a weighted $L^2$-norm and a weighted uniform metric in the so called $L^2_{tem}$, $C_{tem}$ spaces. The new moments estimates of the stochastic convolution with respect to space-time white noise play an important role. In addition, the transportation cost inequalities are also obtained for the stochastic reaction diffusion equations with random initial values.
\end{abstract}



\section{Introduction}
Let $(E, d)$ be a metric space equipped with the  Borel  $\sigma$-field ${\cal B}$. Let $\mathcal{P}(E)$ be the class of all the probability measures on $E$. The $L^p$-Wasserstein distance  between $\mu,\nu\in\mathcal{P}(E)$ is defined by
\begin{align*}
	W_p(\nu, \mu):=\left[\inf \iint_{E\times E}d(x,y)^p\,\pi(\mathrm{d}x,\mathrm{d}y)\right]^{\frac{1}{p}},
\end{align*}
where the infimum is taken over all probability measures $\pi$ on the product space $E\times E$ with marginals $\mu$ and $\nu$. In the study of Monge-Kontorovich optimal transportation problem, this distance is interpreted as the minimal cost to transport
distribution $\nu$ into $\mu$ when the transportation cost 
from $x$ to $y$ is measured by the distance $d$ on $E$.
In applications, to estimate the unknown optimal transportation, we usually need some other quantities to control it, for instance, the relative entropy.
The relative entropy(or Kullback information) of $\nu$ with respect to $\mu$ is defined by
\[H(\nu|\mu):=\int_E \log\left(\frac{\mathrm{d}\nu}{\mathrm{d}\mu}\right)\, \mathrm{d}\nu ,\]
if $\nu$ is absolutely continuous with respect to $\mu$, and $+\infty$ otherwise.
\begin{definition}
	We say that the measure $\mu$ satisfies the $L^p$-transportation cost inequality if there exists a constant $C>0$ such that for all probability measures $\nu$,
	\begin{equation}\label{1.2}
		W_p(\nu, \mu)\leq \sqrt{2C H(\nu| \mu)}.
	\end{equation}
	The case $p=2$ is referred to as the quadratic transportation cost inequality, also Talagrand inequality.
\end{definition}
Consider the following stochastic reaction diffusion equation:
\begin{align}\label{3.1}
	\left\{
	\begin{aligned}
		\frac{\partial u}{\partial t}(t,x)&=\frac{1}{2}\frac{\partial^2 u}{\partial x^2}(t,x)+b(u(t,x))+ \sigma(u(t,x))\frac{\partial^2 W}{\partial t\partial x}(t,x),\ t\in[0,\infty),x\in \mathbb{R},\\
		u(0,x)&=u_0(x), \quad x\in \mathbb{R},
	\end{aligned}
	\right.
\end{align}
where the initial value $u_0$ is a deterministic function, $\frac{\partial^2 W}{\partial t\partial x}(t,x)$ is a space-time white noise on some filtrated probability space $(\Omega, {\cal F}, \{{\cal F}_t\}_{t\geq 0}, \mathbb{P})$, here $\{{\cal F}_t\}_{t\geq 0}$
is the filtration satisfying the usual conditions generated by the white noise. The coefficients
$b(\cdot), \sigma(\cdot): \mathbb{R}\rightarrow \mathbb{R}$ are deterministic
measurable functions.
\vskip 0.3cm
The purpose of this paper is to establish the quadratic transportation cost inequality for the law of the solution of the stochastic reaction diffusion equation (\ref{3.1}) on the so called $L^2_{tem}$, $C_{tem}$ spaces (defined below).
\vskip 0.4cm
In 1996, Talagrand \cite{T1} established the quadratic transportation cost inequality with sharp constant $C=1$ for Gaussian measures on $\mathbb{R}^k$. Since then, transportation cost inequalities and their applications have been widely studied.
The transportation cost inequalities also have close connections with other functional inequalities like log-Sobolev inequality, Poincar\'{e} inequality, the concentration of measure phenomenon, optimal transport problem, and large deviations, see \cite{L,OV,M,T1,T2,BGL,BG,FS,GRS,U,W} and references therein.
\vskip 0.3cm
We note that the transportation cost inequality depends on the underlying metric. The stronger the metric, the stronger the inequality.

\vskip 0.5cm
The Talagrand type transportation cost inequality in \cite{T1} was generalized by
D. Feyel and A. S. \"{U}st\"{u}nel \cite{FU} to the framework of the abstract Wiener space.
It has also been established for the laws of some stochastic processes on the path spaces.
Considering diffusion processes, H. Djellout, A. Guillin and L. Wu \cite{DGW} obtained the quadratic transportation cost inequality under the $L^2([0, T],\mathbb{R}^k)$-metric for stochastic differential equations(SDEs) on $\mathbb{R}^k$. Later,
the quadratic transportation cost inequality under some uniform metric was obtained in \cite{WuZ1} also for SDEs on $\mathbb{R}^k$.
There are now many articles on transportation cost inequalities for various models, see e.g. \cite{WuZ2} for stochastic evolution equations on some Hilbert space $H$ under the metric $L^2([0, T], H)$, \cite{P} for multidimensional semi-martingales including time inhomogeneous diffusions, \cite{Wu,Ma,Sa} for SDEs with pure jump, L\'{e}vy or fractional noises.

\vskip 0.5cm
Concerning stochastic reaction diffusion equations on bounded domains, there exist several works on the  quadratic transportation cost inequality. Let us recall some of them.
B. Boufoussi and S. Hajji \cite{BH} obtained the quadratic transportation cost inequality for stochastic reaction diffusion equations driven by space-time white noise and driven by fractional noise under the $L^2$-metric: $d_2(\xi,\gamma)=(\int_0^T\int_0^1|\xi(t,x)-\gamma(t,x)|^2dtdx)^{\frac{1}{2}}$.
In \cite{KS}, the authors established the quadratic transportation cost inequality for stochastic reaction diffusion equations under the  $L^2$-metric $d_2(\xi,\gamma)=(\int_0^T\int_0^1|\xi(t,x)-\gamma(t,x)|^2dtdx)^{\frac{1}{2}}$ and under the uniform metric $||\xi-\gamma||=\sup_{0\leq t\leq T, 0\leq x\leq 1}|\xi(t,x)-\gamma(t,x)|$ in the case of additive noise.
In \cite{SZ1}, the authors established the quadratic transportation cost inequality for stochastic reaction diffusion equations driven by multiplicative noise under the uniform metric $||\xi-\gamma||$. Later the authors in \cite{SW} obtained the quadratic transportation cost inequality under the uniform norm for stochastic reaction diffusion equations driven by time-white and space-colored Gaussian noise.


\vskip 0.3cm

The purpose of this paper is to establish the quadratic transportation cost inequality for stochastic reaction diffusion equations driven by multiplicative space-time white noise defined on the whole real line $\mathbb{R}$.
Since the space domain is unbounded, we are forced  to work on the so called  $L^2_{tem}$ and $C_{tem}$ spaces (see the definitions  \eqref{defL}, \eqref{defC} below ).
To obtain the quadratic transportation cost inequality, as a crucial step we need to establish  the precise lower order (more difficult than the higher order) moment estimates of stochastic convolution with respect to the space-time white noise under the weighted $L^2$-norm and the weighted supremum norm, which is of independent interest.

\vskip 0.5cm
The rest of the paper is organized as follows. In Section \ref{S:2}, we recall the setup for stochastic reaction diffusion equations and state the main result of the paper. In Section 3 we establish the new moment estimates for stochastic convolutions with respect to the space-time white noise. Section 4 and Section 5 contain the proofs of the quadratic transportation cost inequalities. In Section 6 we prove the quadratic transportation cost inequality for stochastic reaction diffusion equations with random initial conditions. Some inequalities regarding heat kernel are presented in the Appendix.

\vskip 0.5cm
Convention on constants. Throughout the paper the constants denoted by $c_1, c_2,\dots$ are positive and their precise values are not important. The dependence of constants on parameters if needed will be indicated, e.g., $C_T, \Theta_\lambda$.

\section{Statement of the main results}\label{S:2}

\quad In this section, we will recall the setup of  stochastic reaction diffusion equations driven by space-time white noise and state the main result of the paper.
Let $C_0^2(\mathbb{R})$ denote the set of continuous functions with compact supports whose second derivatives are also  continuous.
We say that an adapted,
continuous random field $\{u(t,x): (t,x)\in [0,\infty)\times \mathbb{R}\}$ is a solution to the stochastic partial differential equation (SPDE) (\ref{3.1}) if for $t\geq 0$,
$\phi \in C_0^2(\mathbb{R})$
\begin{align*}
	&\int_{\mathbb{R}} u(t,x)\phi(x)\mathrm{d}x=\int_{\mathbb{R}} u_0(x)\phi(x)\mathrm{d}x
	+\frac{1}{2}\int_{0}^{t} \mathrm{d}s\int_{\mathbb{R}} u(s,x)\phi^{\prime\prime}(x)\mathrm{d}x\nonumber\\
	&+\int_{0}^{t}\mathrm{d}s\int_{\mathbb{R}} b(u(s,x))\phi(x)\mathrm{d}x+ \int_{0}^{t}\int_{\mathbb{R}}\sigma(u(s,x))\phi(x)W(
	\mathrm{d}s,\mathrm{d}x),\ \mathbb{P}\text{-a.s.}
\end{align*}

It was shown in \cite{S} that $u$ is a solution to SPDE (\ref{3.1}) if and only if $u$ satisfies the following integral equation
\begin{align}\label{3.3}
	u(t,x)=&P_tu_0(x)+\int_{0}^{t}\int_{\mathbb{R}} p_{t-s}(x,y)b(u(s,y))\mathrm{d}y\mathrm{d}s\nonumber\\
	&+ \int_{0}^{t}\int_{\mathbb{R}} p_{t-s}(x,y)\sigma(u(s,y))W(\mathrm{d}s,\mathrm{d}y),\ \mathbb{P}\text{-a.s.}
\end{align}
where $p_{t}(x,y):=\frac{1}{\sqrt{2\pi t}} e^{-\frac{(x-y)^2}{2t}}$, and $\{P_t\}_{t\geq 0}$ is the corresponding heat semigroup on $\mathbb{R}$.

\vskip 0.5cm

Next we introduce some spaces used in the paper.
For given $\lambda>0$, the space
\begin{align*}
	L^2_\lambda:=\left\{f:\mathbb{R}\to\mathbb{R}\text{ is measurable and } \int_{\mathbb{R}} |f(x)|^2e^{-2\lambda|x|}\mathrm{d}x<\infty\right\}
\end{align*}
equipped with the inner product
\begin{align*}
	\langle f,g\rangle_{L^2_\lambda}=\int_{\mathbb{R}} |f(x)g(x)|e^{-2\lambda|x|}\mathrm{d}x,\ f,g\in L^2_\lambda
\end{align*}
is a Hilbert space, and denote the induced norm by $\| \cdot \|_{L^2_{\lambda}}$.
Define
\begin{align*}
	E_\lambda:=\left\{f\in C(\mathbb{R}):\sup_{x\in\mathbb{R}}\left( |f(x)|e^{-\lambda|x|}\right)<\infty \right\},
\end{align*}
for given $\lambda>0$, and equipped with the metric
\begin{align*}
	\varrho_\lambda(f,g):=\sup_{x\in\mathbb{R}}\left( |f(x)-g(x)|e^{-\lambda|x|}\right),\ f,g\in E_\lambda,
\end{align*}
$E_\lambda$ is a Polish space. We also recall that the $L^2_{tem}$ and $C_{tem}$ spaces are defined by
\begin{align}
	L^2_{tem}:&=\left\{f\in C(\mathbb{R}): \int_{\mathbb{R}} |f(x)|^2e^{-2\lambda |x|}\mathrm{d}x<\infty \text{ for all } \lambda>0 \right\},\label{defL}\\
	C_{tem}:&=\left\{f\in C(\mathbb{R}): \sup_{x\in\mathbb{R}}\left(|f(x)|e^{-\lambda |x|}\right)<\infty \text{ for all } \lambda>0 \right\},\label{defC}
\end{align}
and endowed with the metrics respectively defined by
\begin{align*}
	\rho(f,g):&=\sum_{n=1}^{\infty}\frac{1}{2^n}\min\left\{1, \|f-g\|_{L^2_{1/n}}\right\},\ f,g\in L^{2}_{tem}, \\
	\varrho(f,g):&=\sum_{n=1}^{\infty}\frac{1}{2^n}\min\left\{1, \varrho_{1/n}(f,g)\right\},\ f,g\in C_{tem}.
\end{align*}
Then $f_n\to f$ in $L^2_{tem}$ iff for every $\lambda>0$, $\|f_n-f\|_{L^2_\lambda}\to0$, and $f_n\rightarrow f$ in $C_{tem}$ iff for every $\lambda>0$, $\varrho_{\lambda}(f_n,f)\rightarrow 0$. It is known that $C_{tem}$ and $L^2_{tem}$ are Polish spaces.

\vskip 0.5cm
Let $(E, d)$ be a Polish space.
Consider a continuous Markov process on $E$ with a given transition kernel $P_{t}(x, \cdot)$. For  $T>0$ and $\mu \in \mathcal{P}(E)$, let $P^{\mu}$ denote the distribution of the Markov process  with initial distribution $\mu$, i.e., $P^{\mu}$ is the unique probability measure on the path space
\begin{align*}
	C([0, T],E)\text{ equipped with the metric } \tilde{d}(\varphi, \psi):=\sup_{t \in[0, T]} d\left(\varphi(t), \psi(t)\right),
\end{align*}
When $\mu=\delta_{x}$, the Dirac measure at $x \in E$, we simply denote $P^{\mu}=P^{x}$. Then
\begin{align*}
	P^{\mu}=\int_{E} P^{x} \mu(\mathrm{d} x), \quad \mu \in \mathcal{P}(E).
\end{align*}

\vskip 0.5cm
We introduce the following hypotheses.
\begin{itemize}
	\item[({\bf H1})] There exists a constant $L_b$ such that for all $x,y\in \mathbb{R}$,
	\begin{align*}
		|b(x)-b(y)| \leq\, & L_b |x-y|.
	\end{align*}
	\item[({\bf H2})] There exist constants $K_{\sigma}$ and $L_{\sigma}$ such that for all $x,y\in \mathbb{R}$,
	\begin{align*}
		&({\bf H2(a)})  &|\sigma(x)|\leq&\,  K_{\sigma}, \\
		&({\bf H2(b)}) &|\sigma(x)-\sigma(y)| \leq\, &L_{\sigma}|x-y|.
	\end{align*}
\end{itemize}

It is well known that under the hypotheses ({\bf H1}) and ({\bf H2}), SPDE (\ref{3.1}) admits a unique random field solution $u(t,x)$. In fact, for the existence and uniqueness the diffusion coefficient $\sigma(\cdot)$ does not need to be bounded, the stronger assumption ({\bf H2}) is needed for proving the transportation cost inequality.
\begin{proposition}[\cite{S}]\label{exist2}
	Assume that (\textbf{H1}) and  (\textbf{H2(b)}) hold and $u_0\in C_{tem}$. Then there exists a random field solution to the stochastic reaction diffusion equation \eqref{3.1} with sample paths a.s. in $C([0,T], C_{tem})$.
\end{proposition}
\begin{proposition}\label{exist1}
	Assume that (\textbf{H1}) and  (\textbf{H2(b)}) hold and $u_0\in L^2_{tem}$. Then there exists a random field solution to the stochastic reaction diffusion equation \eqref{3.1} with sample paths a.s. in $C([0,T], L^2_{tem})$.
\end{proposition}
The proof of Proposition \ref{exist1} is similar to that of Proposition \ref{exist2}. We omit the details.

Here are the main results.
\begin{theorem}\label{Lresult}
	Suppose the hypotheses {\rm({\bf H1})} and {\rm({\bf H2})} hold. If $u_0\in L^2_{tem}$, then, the law of the solution $u(\cdot, \cdot)$ of SPDE (\ref{3.1}) satisfies the quadratic transportation cost inequality \eqref{1.2} on the space $C([0, T],L^2_{tem})$.
\end{theorem}

\begin{theorem}\label{Ctemresult}
	Suppose the hypotheses {\rm({\bf H1})} and {\rm({\bf H2})} hold.
	If $u_0\in C_{tem}$, then, the law of the solution $u(\cdot, \cdot)$ of SPDE (\ref{3.1}) satisfies the quadratic transportation cost inequality \eqref{1.2} on the space $C([0, T],C_{tem})$.
\end{theorem}

Using the approach in \cite{WZ} we derive the following transportation inequality for the stochastic reaction diffusion equation with random initial values.
\begin{corollary}\label{Lrandom}
	Suppose the hypotheses {\rm({\bf H1})} and {\rm({\bf H2})} hold, and $\mu\in \mathcal{P}(L^2_{tem})$.Then
	\begin{align*}
		W_2(Q,P^{\mu})^2\leq c_1 H(Q|P^{\mu}),\quad Q\in \mathcal{P}(C([0,T],L^2_{tem}))
	\end{align*}
	holds for some constant $c_1>0$ if and only if
	\begin{align*}
		W_2(\nu, \mu)^2\leq \tilde{c}_1H(\nu|\mu),\quad \nu\in\mathcal{P}(C_{tem})
	\end{align*}
	holds for some constant $\tilde{c}_1>0$.
\end{corollary}

\begin{corollary}\label{LCtemrandom}
	Suppose the hypotheses {\rm({\bf H1})} and {\rm({\bf H2})} hold, and $\mu\in \mathcal{P}(L^2_{tem}\cap C_{tem})$.Then
	\begin{align*}
		W_2(Q,P^{\mu})^2\leq c_2 H(Q|P^{\mu}),\quad Q\in \mathcal{P}(C([0,T],L^2_{tem}\cap C_{tem}))
	\end{align*}
	holds for some constant $c_2>0$ if and only if
	\begin{align*}
		W_2(\nu, \mu)^2\leq \tilde{c}_2 H(\nu|\mu),\quad \nu\in\mathcal{P}(L^2_{tem}\cap C_{tem})
	\end{align*}
	holds for some constant $\tilde{c}_2>0$.
\end{corollary}

\skip 0.5cm

Let $E=L^2_{tem}$or $E=C_{tem}$.
Let $\mu$ be the law of the random field solution $u(\cdot, \cdot)$ of SPDE (\ref{3.1}), viewed as a probability measure on $C([0, T], E)$. For the proof of Theorem \ref{Lresult} and \ref{Ctemresult}, we recall  a lemma describing the probability measure $\nu$ that is absolutely continuous with respect to $\mu$.

Let $\nu\ll \mu$ on $C([0, T],E)$.
Define a new probability measure $\mathbb{Q}$ on the filtered probability space $(\Omega, {\cal F}, \{{\cal F}_{t}\}_{0\leq t\leq T}, \mathbb{P})$ by
\begin{align}\label{add 0303.1}
	\mathrm{d}\mathbb{Q}:=\frac{\mathrm{d}\nu}{\mathrm{d}\mu}(u) \,\mathrm{d}\mathbb{P} .
\end{align}
Denote the Radon-Nikodym derivative restricted on ${\cal F}_t$ by
\[M_t:=\left. \frac{\mathrm{d}\mathbb{Q}}{\mathrm{d}\mathbb{P}}\right |_{{\cal F}_t}, \quad t\in [0, T].\]
Then $M_t, t\in [0, T]$ forms a $\mathbb{P}$-martingale.

The following Girsanov transformation lemma can be found in  \cite{KS}.
\begin{lemma}\label{CtemGirsa}
	There exists an adapted random field $h=\{h(s,x), (s,x)\in [0, T]\times \mathbb{R}\}$ such that $\mathbb{Q}\text{-a.s.}$ for all $t\in [0, T]$,
	\begin{align*}
		\int_{0}^{t}\int_{\mathbb{R}} | h(s,x)|^2 \,\mathrm{d}x\mathrm{d}s <\infty
	\end{align*}
	and
	\begin{align*}
		\widetilde{W}(\mathrm{d}t,\mathrm{d}x):=W(\mathrm{d}t,\mathrm{d}x)-h(t,x)\,\mathrm{d}x\mathrm{d}t,
	\end{align*}
	is a space-time white noise on $[0,T]\times\mathbb{R}$. Moreover,
	\begin{align*}
		M_t=\exp\left (\int_{0}^{t}\int_{\mathbb{R}} h(s,x)\,W(\mathrm{d}s,\mathrm{d}x)-\frac{1}{2}\int_{0}^{t}\int_{\mathbb{R}} | h(s,x)|^2\,\mathrm{d}x\mathrm{d}s\right ), \quad \mathbb{Q}\text{-a.s.},
	\end{align*}
	and
	\begin{align}\label{4.4}
		H(\nu|\mu)=\frac{1}{2}\mathbb{E}^{\mathbb{Q}}\left[\int_{0}^{t}\int_{\mathbb{R}} | h(s,x)|^2\,\mathrm{d}x\mathrm{d}s\right],
	\end{align}
	where $\mathbb{E}^{\mathbb{Q}}$ stands for the expectation under the measure $\mathbb{Q}$.
\end{lemma}


\section{Moment estimates}
In the proof of the main results, the moments estimates (see Lemma \ref{moment1} and \ref{moment2}) for stochastic convolution against space-time white noise will play a crucial role. Before we give the estimates, we prove the Burkholder-type inequality for Hilbert-space-valued stochastic integral against space-time white noise.

\vskip 0.5cm
We present some estimates associated with the heat kernal $p_t(x,y)=\frac{1}{\sqrt{2\pi t}}e^{-\frac{(x-y)^2}{2t}}$, whose proofs are straightforward so we omit them here. For any $x\in \mathbb{R}$, $t>0$ and $\eta \in \mathbb{R}$,
\begin{align}
	&\int_{\mathbb{R}} p_{t}(x, y) e^{\eta|y|} \mathrm{d} y \leq 2 e^{\frac{\eta^{2} t}{2}} e^{\eta|x|},\label{103.1}\\
	&\int_{\mathbb{R}} p_{t}(x, y)^2 e^{\eta|y|} \mathrm{d} y \leq \frac{1}{\sqrt{\pi t}} e^{\frac{\eta^{2} t}{4}} e^{\eta|x|}.\label{103.2}
\end{align}
For a function $f\in L^2_\lambda$, $\lambda>0,\ t>0$,
\begin{align}
	\label{kernal} \left\Vert P_tf\right\Vert_{L^2_{\lambda}} \leq \sqrt{2}e^{\lambda^2 t}\Vert f\Vert_{L^2_{\lambda}}.
\end{align}

\begin{lemma}[Burkholder-Davis-Gundy's inequality]\label{BDG}
	Let $H$ be a Hilbert space with inner product and norm denoted by $\langle \cdot,\cdot\rangle_H$ and $|\cdot|_H$, respectively. $\Phi(t,x)$ is a $H$-valued adapted random field such that the following stochastic integral with respect to the space-time white noise is well-defined.  Assume that for $p>0$,
	\begin{align*}
		\mathbb{E}\left(\int_{0}^{t}\int_{\mathbb{R}}|\Phi(r,z)|_H^2\mathrm{d}z\mathrm{d}r\right)^{\frac{p}{2}}<\infty.
	\end{align*}
	Then there exists a constant $c_p>0$ such that for $t>0$,
	\begin{align*}
		\mathbb{E}\sup_{s\in[0, t]} \left|\int_{0}^{s}\int_{\mathbb{R}}\Phi(r,z)W(\mathrm{d}r,\mathrm{d}z)\right|_H^p\leq c_p~ \mathbb{E}\left(\int_{0}^{t}\int_{\mathbb{R}}|\Phi(r,z)|_H^2\mathrm{d}z\mathrm{d}r\right)^{\frac{p}{2}}.
	\end{align*}	
\end{lemma}
\begin{proof}
	It is known (see \cite[Chapter 4]{DZ}, also \cite{DS}) that Walsh's stochastic integral against the space-time white noise can be formulated as an It$\mathrm{\hat{o}}$ integral w.r.t. a cylindrical Wiener process. Let us make it precise in the current setting.
	
	For $t\geq 0$, define an operator $\Phi(t):L^2(\mathbb{R})\rightarrow H$ by
	\begin{align*}
		[\Phi(t)](f):=\int_{\mathbb{R}} \Phi(t,x)f(x)\mathrm{d}x,
	\end{align*}
	for $f\in L^2(\mathbb{R})$. Then,  $\Phi(t)$ is a Hilbert-Schmidt operator, i.e. $\Phi(t)\in L_2(L^2(\mathbb{R}),H)$ and
	$$|\Phi(t)|_{L_2(L^2(\mathbb{R}),H)}^2=\int_{\mathbb{R}}|\Phi(t,x)|_H^2\mathrm{d}x.$$
	Indeed, let $\{e_n\}_{n\geq 1}$ be an orthonormal basis of $L^2(\mathbb{R})$, and if $\{g_j\}_{j\geq 1}$ is an orthonormal basis of $H$ we have that
	\begin{align*}
		|\Phi(s)|_{L_2(L^2(\mathbb{R}),H)}^2=&\sum_{n=1}^{\infty}|\Phi(s)e_n|_H^2 \\
		= & \sum_{n=1}^{\infty}\left|\int_{\mathbb{R}}\Phi(s,y)e_n(y)\mathrm{d}y \right|_H^2 \\
		= & \sum_{n=1}^{\infty}\sum_{j=1}^{\infty} \left\langle \int_{\mathbb{R}}\Phi(s,y)e_n(y)\mathrm{d}y ,g_j\right\rangle_H^2 \\
		= & \sum_{j=1}^{\infty}\sum_{n=1}^{\infty} \left(\int_{\mathbb{R}} \langle\Phi(s,y),g_j\rangle_H ~e_n(y)\mathrm{d}y\right)^2 \\
		= & \sum_{j=1}^{\infty} \int_{\mathbb{R}}\langle\Phi(s,y),g_j\rangle_H^2\mathrm{d}y \\
		= & \int_{\mathbb{R}} |\Phi(s,y)|_H^2\mathrm{d}y.
	\end{align*}
	
	Now,  for $r\geq 0$, we expand $\Phi(r,\cdot)\in L^2(\mathbb{R})$ in the basis $\{e_n\}_{n\geq 1}$ and write
	\begin{align*}
		&\int_{0}^{s}\int_{\mathbb{R}}\Phi(r,z)W(\mathrm{d}r,\mathrm{d}z) \\
		= & \int_{0}^{s}\int_{\mathbb{R}} \sum_{n=1}^{\infty} \langle \Phi(r,\cdot),e_n\rangle_{L^2(\mathbb{R})} e_n(z)W(\mathrm{d}r,\mathrm{d}z) \\
		= & \int_{0}^{s} \sum_{n=1}^{\infty} \langle \Phi(r,\cdot),e_n\rangle_{L^2(\mathbb{R})} \mathrm{d}\beta_n(r) \\
		= & \int_{0}^{s} \sum_{n=1}^{\infty} \Phi(r)e_n \,\mathrm{d}\beta_n(r) \\
		= & \int_{0}^{s} \Phi(r)\mathrm{d} W_r,
	\end{align*}
	where $\beta_n(t):=\int_{0}^{t}\int_{\mathbb{R}}e_n(z)W(\mathrm{d}r,\mathrm{d}z),\ t\geq 0,n\geq 1$ are independent $1$-dimensional Brownian motions, and $W_r:=\sum_{n=1}^{\infty} e_n \beta_n(r),\ r\geq 0$ is a $L^2(\mathbb{R})$-cylindrical Wiener process.
	
	According to \cite[Theorem 4.36]{DZ}, there exists a constant $c_p$ such that
	\begin{align*}
		&\mathbb{E}\sup_{s\in[0, t]} \left|\int_{0}^{s}\int_{\mathbb{R}}\Phi(r,z)W(\mathrm{d}r,\mathrm{d}z)\right|_H^p \\
		=&\mathbb{E}\sup_{s \in[0, t]}\left|\int_{0}^{s} \Phi(r)\mathrm{d} W_r\right|_H^p \\
		\leq & c_p~ \mathbb{E}\left(\int_{0}^{t}|\Phi(s)|_{L_2(L^2(\mathbb{R}),H)}^2\mathrm{d}s\right)^{\frac{p}{2}} \\
		= &c_p~\mathbb{E}\left(\int_{0}^{t}\int_{\mathbb{R}} |\Phi(s,y)|_H^2\mathrm{d}y\mathrm{d}s\right)^{\frac{p}{2}}.
	\end{align*}
\end{proof}

\vskip 0.5cm
Recall the Hilbert space $L^2_{\lambda}$ defined in Section \ref{S:2} with inner product $\langle\cdot,\cdot\rangle_{L^2_{\lambda}}$ and induced norm $\Vert\cdot\Vert_{L^2_{\lambda}}$.

\begin{lemma}\label{moment1}
	Let $\{\sigma(s,y):(s, y)\in [0,T]\times \mathbb{R}\}$ be an adapted random field such that the following stochastic convolution with respect to the space-time white noise is well defined and $\sigma(t,\cdot)\in L^2_{\lambda}$ for  every $t$, $\mathbb{P}$-a.s. Then for any $p>8,\ T>0$, there exists a constant $C_{\lambda,T,p}$ increasing w.r.t. $\lambda$  such that
	\begin{align}\label{mm1}
		&\mathbb{E} \sup_{t \leq T}\left\{ \int_{\mathbb{R}}\left|\int_{0}^{t} \int_{\mathbb{R}} p_{t-s}(x, y) \sigma(s, y) W(\mathrm{d}s,\mathrm{d}y)\right|^2 e^{-2\lambda|x|} \mathrm{d}x\right\}^{\frac{p}{2}} \nonumber\\
		\leq & C_{\lambda, T,p} ~\mathbb{E} \int_{0}^{T}\Vert\sigma(r,\cdot)\Vert_{L^2_{\lambda}}^p \mathrm{d}r.
	\end{align}
\end{lemma}
\begin{proof}
	We adopt the factorization method(see e.g.\cite{DZ}). For $0<\alpha<\frac{1}{8}$, set
	\begin{align}
		J_\sigma(s,y):&=\int_{0}^{s}\int_{\mathbb{R}}(s-r)^{-\alpha}p_{s-r}(y,z)\sigma(r,z)W(\mathrm{d}r,\mathrm{d}z),\label{factor1} \\
		J^{\alpha-1}f(t,x):&=\frac{\sin (\pi \alpha)}{\pi} \int_{0}^{t}\int_{\mathbb{R}}(t-s)^{\alpha-1}p_{t-s}(x,y)f(s,y)\mathrm{d}s\mathrm{d}y.\label{factor2}
	\end{align}
	Then we have
	\begin{equation}\label{mm11}
		\int_{0}^{t} \int_{\mathbb{R}} p_{t-s}(x, y) \sigma(s, y) W(\mathrm{d}s,\mathrm{d}y)=J^{\alpha-1}\left(J_\alpha \sigma\right)(t,x).
	\end{equation}
	Now we consider the left-hand side of \eqref{mm1}. Using (\ref{mm11}), Minkowski's  inequality, \eqref{kernal} and H\"{o}lder's inequality, we have for $p>8$
	\begin{align}\label{mm12}
		&\mathbb{E}~\sup_{t\leq T} \left\{\int_{\mathbb{R}}\left|\int_{0}^{t} \int_{\mathbb{R}} p_{t-s}(x, y) \sigma(s, y) W(\mathrm{d}s,\mathrm{d}y)\right|^2e^{-2\lambda |x|}\mathrm{d}x\right\}^\frac{p}{2} \nonumber\\
		= & \mathbb{E}~\sup_{t\leq T} \left\Vert\int_{0}^{t} \int_{\mathbb{R}} p_{t-s}(\cdot, y) \sigma(s, y) W(\mathrm{d}s,\mathrm{d}y)\right\Vert_{L^2_\lambda}^p \nonumber\\
		\leq & \left|\frac{\sin (\pi \alpha)}{\pi}\right|^p\mathbb{E}~\sup_{t\leq T} \left\Vert\int_{0}^{t} \int_{\mathbb{R}} (t-s)^{\alpha-1}p_{t-s}(\cdot, y) J_\alpha\sigma(s, y) \mathrm{d}y\mathrm{d}s\right\Vert_{L^2_\lambda}^p \nonumber\\
		\leq & \pi^{-p}~\mathbb{E}~\sup_{t\leq T}\left\{ \int_{0}^{t}(t-s)^{\alpha-1}\left\Vert \int_{\mathbb{R}}p_{t-s}(\cdot, y) J_\alpha\sigma(s,y)\mathrm{d}y\right\Vert_{L^2_\lambda}\mathrm{d}s\right\}^p \nonumber\\
		\leq & \pi^{-p} 2^{\frac{p}{2}}e^{p\lambda^2 T}~\mathbb{E}~\sup_{t \leq T}\left\{\int_{0}^{t}(t-s)^{\alpha-1} \Vert J_\alpha\sigma(s,\cdot)\Vert_{L^2_\lambda}\mathrm{d}s\right\}^p \nonumber\\
		\leq & \pi^{-p} 2^{\frac{p}{2}}e^{p\lambda^2 T} ~ \left(\int_{0}^{T} s^{(\alpha-1)\frac{p}{p-1}}\mathrm{d}s\right)^{p-1}  \int_{0}^{T}\mathbb{E}\left\Vert J_\alpha \sigma(s,\cdot)\right\Vert_{L^2_{\lambda}}^p \mathrm{d}s \nonumber\\
		\leq & C'_{\lambda,T,p}  \int_{0}^{T}\mathbb{E}\left\Vert J_\alpha \sigma(s,\cdot)\right\Vert_{L^2_{\lambda}}^p \mathrm{d}s,
	\end{align}
	where the condition $\alpha>\frac{1}{p}$ is used for a finite integral.
	By \eqref{factor1} and Lemma \ref{BDG} we have
	\begin{align*}		
		&\mathbb{E}\Vert J_\alpha \sigma(s,\cdot)\Vert_{L^2_{\lambda}}^p \nonumber\\
		= &\mathbb{E} \left\Vert\int_{0}^{s}\int_{\mathbb{R}}(s-r)^{-\alpha}p_{s-r}(\cdot,z)\sigma(r,z)W(\mathrm{d}r,\mathrm{d}z)\right\Vert_{L^2_{\lambda}}^p\nonumber\\
		\leq &  c_p \mathbb{E}\left(\int_0^s\int_{\mathbb{R}} (s-r)^{-2 \alpha}\Vert p_{s-r}(\cdot,z)\sigma(r,z)\Vert^2_{L^2_{\lambda}} \mathrm{d}z\mathrm{d} r\right)^{\frac{p}{2}}.
	\end{align*}
	With this estimate, \eqref{mm12} is bounded by
	\begin{align}\label{mm13}
		&\mathbb{E}~\sup_{t\leq T} \left\{\int_{\mathbb{R}}\left|\int_{0}^{t} \int_{\mathbb{R}} p_{t-s}(x, y) \sigma(s, y) W(\mathrm{d}s,\mathrm{d}y)\right|^2e^{-2\lambda |x|}\mathrm{d}x\right\}^\frac{p}{2} \nonumber\\
		\leq & C''_{\lambda,T,p}~\int_{0}^{T}\mathbb{E}\left(\int_0^s\int_{\mathbb{R}}(s-r)^{-2 \alpha}\Vert p_{s-r}(\cdot,z)\sigma(r,z)\Vert^2_{L^2_{\lambda}}\mathrm{d}z \mathrm{d} r\right)^{\frac{p}{2}}\mathrm{d}s.
	\end{align}
	
	We continue to estimate the right-hand side of \eqref{mm13}. By \eqref{103.2} and H\"{o}lder's inequality we have
	\begin{align}\label{mm14}
		&\int_{0}^{T}\mathbb{E}\left(\int_0^s\int_{\mathbb{R}}(s-r)^{-2 \alpha}\Vert p_{s-r}(\cdot,z)\sigma(r,z)\Vert^2_{L^2_{\lambda}}\mathrm{d}z\mathrm{d} r\right)^{\frac{p}{2}}\mathrm{d}s \nonumber\\
		= & \int_{0}^{T}\mathbb{E}\left(\int_0^s\int_{\mathbb{R}}(s-r)^{-2 \alpha}\int_{\mathbb{R}}| p_{s-r}(y,z)\sigma(r,z)|^2e^{-2\lambda|y|}\mathrm{d}y\,\mathrm{d}z\mathrm{d} r\right)^{\frac{p}{2}}\mathrm{d}s   \nonumber\\
		\leq & \mathbb{E}\int_{0}^{T}  \left(\int_{0}^{s}\int_{\mathbb{R}}(s-r)^{-2\alpha-\frac{1}{2}} \pi ^{-\frac{1}{2}}e^{\lambda^2s}e^{-2\lambda|z|} |\sigma(r,z)|^2\mathrm{d}z\mathrm{d}r\right)^{\frac{p}{2}}\mathrm{d}s \nonumber\\
		\leq & \pi^{-\frac{p}{4}}e^{\frac{p}{2}\lambda^2T}~\mathbb{E}\int_{0}^{T}  \left(\int_{0}^{s}(s-r)^{-(2\alpha+\frac{1}{2})\frac{p}{p-2}}\mathrm{d}r\right)^{\frac{p-2}{2}} \left(\int_{0}^{s}\Vert\sigma(r,\cdot)\Vert^{2\times\frac{p}{2}}_{L^2_{\lambda}}\mathrm{d}r\right) \mathrm{d}s \nonumber\\
		\leq&2^{-\frac{p-2}{2}}\pi^{-\frac{p}{4}}e^{\frac{p}{2}\lambda^2 T} \left(\frac{p-2}{p(\frac{1}{4}-\alpha)-1}\right)^{\frac{p-2}{2}} ~\mathbb{E}\int_{0}^{T}s^{p(\frac{1}{4}-\alpha)-1} \int_{0}^{s}\Vert\sigma(r,\cdot)\Vert_{L^2_{\lambda}}^p \mathrm{d}r \,\mathrm{d}s \nonumber\\
		\leq & 2^{-\frac{p-2}{2}}\pi^{-\frac{p}{4}}e^{\frac{p}{2}\lambda^2 T} T^{p(\frac{1}{4}-\alpha)} \left(\frac{p-2}{p(\frac{1}{4}-\alpha)-1}\right)^{\frac{p-2}{2}} \frac{1}{p(\frac{1}{4}-\alpha)}~\mathbb{E}\int_{0}^{T}\Vert\sigma(r,\cdot)\Vert_{L^2_{\lambda}}^p \mathrm{d}r\nonumber\\
		=&:C'''_{\lambda,T,p}~\mathbb{E} \int_{0}^{T}\Vert\sigma(r,\cdot)\Vert_{L^2_{\lambda}}^p \mathrm{d} r.
	\end{align}
	In the last step we changed the order of integration, and the condition $\alpha<\frac{1}{4}-\frac{1}{p}$ is used.
	
	Let
	\begin{align*}
		C_{\lambda,T,p}:=\min_{\frac{1}{p}<\alpha<\frac{1}{4}-\frac{1}{p}} C''_{\lambda,T,p}\cdot C'''_{\lambda,T,p}.
	\end{align*}
	Then \eqref{mm1} follows from \eqref{mm13} and \eqref{mm14}.
\end{proof}

\begin{lemma}\label{moment2}
	Let $\{\sigma(s,y):(s, y)\in [0,T]\times \mathbb{R}\}$ be an adapted random field such that the following stochastic convolution with respect to the space-time white noise is well defined and $\sigma(t,\cdot)\in L^2_{\lambda}$ for  every $t$, $\mathbb{P}$-a.s.. Then for any $0<p\leq 8,\ T,\epsilon>0$, there exists a constant $C_{\epsilon,\lambda,T,p}$  increasing w.r.t. $\lambda$ such that
	\begin{align*}
		&\mathbb{E} \sup_{t \leq T}\left\{ \int_{\mathbb{R}}\left|\int_{0}^{t} \int_{\mathbb{R}} p_{t-s}(x, y) \sigma(s, y) W(\mathrm{d}s,\mathrm{d}y)\right|^2 e^{-2\lambda|x|} \mathrm{d}x\right\}^{\frac{p}{2}}	\nonumber\\
		\leq &\epsilon~\mathbb{E} \sup_{t\leq T}\Vert\sigma(t,\cdot)\Vert_{L^2_{\lambda}}^p+C_{\epsilon,\lambda,T,p}~\mathbb{E} \int_{0}^{T}\Vert\sigma(t,\cdot)\Vert_{L^2_{\lambda}}^p \mathrm{d}t,
	\end{align*}
\end{lemma}
\begin{proof}
	The proof is inspired by \cite{SZ1}.
	
	Step 1: We show that for $q>8,\ \eta>0$,
	\begin{align}\label{mm21}
		&\mathbb{P}\left(\sup_{t \leq T}\left\Vert\int_{0}^{t} \int_{\mathbb{R}} p_{t-s}(x, y) \sigma(s, y) W(\mathrm{d}s,\mathrm{d}y)\right\Vert_{L^2_{\lambda}}>\eta\right) \nonumber\\
		\leq& \mathbb{P}\left(\int_{0}^{T}\Vert\sigma(s,\cdot)\Vert_{L^2_{\lambda}}^q \mathrm{d} s>\eta^q\right)+\frac{c_{\lambda,T, q}}{\eta^q} \mathbb{E} \min \left\{\eta^q, \int_{0}^{T}\Vert\sigma(s,\cdot)\Vert_{L^2_{\lambda}}^q \mathrm{d}s\right\}.
	\end{align}
	Set
	\begin{align*}
		\Omega_\eta:=\left\{\omega:\int_{0}^{T}\left\Vert\sigma(r,\cdot)\right\Vert^q_{L^2_\lambda}\mathrm{d}r\leq \eta^q\right\}.
	\end{align*}
	Then we have by Chebyshev's inequality that
	\begin{align}\label{mm22}
		&\mathbb{P}\left(\sup_{t \leq T}\left\Vert\int_{0}^{t} \int_{\mathbb{R}} p_{t-s}(\cdot, y) \sigma(s, y) W(\mathrm{d}s,\mathrm{d}y)\right\Vert_{L^2_{\lambda}}>\eta\right) \nonumber\\
		\leq & \mathbb{P}(\Omega/\ \Omega_\eta)+\mathbb{P}\left(\sup_{t \leq T}\left[\left\Vert\int_{0}^{t} \int_{\mathbb{R}} p_{t-s}(\cdot, y) \sigma(s, y) W(\mathrm{d}s,\mathrm{d}y)\right\Vert_{L^2_\lambda} \mathbbm{1}_{\Omega_\eta}\right]>\eta\right) \nonumber\\
		\leq & \mathbb{P}\left(\int_{0}^{T}\Vert\sigma(s,\cdot)\Vert_{L^2_{\lambda}}^q \mathrm{d} s>\eta^q\right)+\frac{1}{\eta^q}\mathbb{E}\sup_{t \leq T}\left[\left\Vert\int_{0}^{t} \int_{\mathbb{R}} p_{t-s}(\cdot, y) \sigma(s, y) W(\mathrm{d}s,\mathrm{d}y)\right\Vert_{L^2_\lambda} \mathbbm{1}_{\Omega_\eta}\right]^q.
	\end{align}	
	Define a new random field
	\begin{align*}
		\tilde{\sigma}(s,y):=\sigma(s,y)\mathbbm{1}_{\{\omega:\int_{0}^{s}\Vert \sigma(r,\cdot)\Vert^q_{L^2_\lambda}\mathrm{d}r\leq \eta^q\}}.
	\end{align*}	
	By the local property of the stochastic integral(see Lemma A.1 in \cite{SZ1}),
	\begin{align*}
		\mathbbm{1}_{\Omega_\eta}\int_{0}^{t}\int_{\mathbb{R}}p_{t-s}(x,y)\sigma(s,y)W(\mathrm{d}s,\mathrm{d}y)=\mathbbm{1}_{\Omega_\eta}\int_{0}^{t}\int_{\mathbb{R}}p_{t-s}(x,y)\tilde{\sigma}(s,y)W(\mathrm{d}s,\mathrm{d}y),\ \mathbb{P}\text{-a.s.}
	\end{align*}
	In view of the above equality, we can use the bound in Lemma \ref{moment1}:
	\begin{align}\label{mm23}
		&\mathbb{E}\sup_{t \leq T}\left[\left\Vert\int_{0}^{t} \int_{\mathbb{R}} p_{t-s}(\cdot, y) \sigma(s, y) W(\mathrm{d}s,\mathrm{d}y)\right\Vert_{L^2_\lambda} \mathbbm{1}_{\Omega_\eta}\right]^q\nonumber\\
		\leq& \mathbb{E}\sup_{t \leq T}\left\Vert\int_{0}^{t} \int_{\mathbb{R}} p_{t-s}(\cdot, y) \tilde{\sigma}(s, y) W(\mathrm{d}s,\mathrm{d}y)\right\Vert_{L^2_\lambda}^q\nonumber\\
		\leq & C_{\lambda,T,q} ~\mathbb{E}\int_{0}^{T}\Vert\tilde{\sigma}(s,\cdot) \Vert^q_{L^2_{\lambda}}\mathrm{d}s\nonumber\\
		\leq & C_{\lambda,T,q} ~\mathbb{E}\min\{\eta^q,\int_{0}^{T}\Vert\sigma(s,\cdot) \Vert^q_{L^2_{\lambda}}\mathrm{d}s\},
	\end{align}
	where $C_{\lambda,T,q}$ is the constant appeared in \eqref{mm1} with $q$ replacing $p$. Now \eqref{mm21} follows immediately from  \eqref{mm22} and \eqref{mm23}.
	
	Step 2: For $0<p\leq 8$, we  follow Lemma A.2 in \cite{SZ1} to obtain that
	\begin{align*}
		&\mathbb{E} ~\sup_{t \leq T}\left\Vert\int_{0}^{t} \int_{\mathbb{R}} p_{t-s}(\cdot, y) \sigma(s, y) W(\mathrm{d}s,\mathrm{d}y)\right\Vert_{L^2_{\lambda}}^p\nonumber\\
		=&\int_{0}^{\infty} p\eta^{p-1}\mathbb{P}\left(\sup_{t \leq T}\left\Vert \int_{0}^{t} \int_{\mathbb{R}} p_{t-s}(\cdot, y) \sigma(s, y) W(\mathrm{d}s,\mathrm{d}y)\right\Vert^p_{L^2_\lambda}>\eta\right)\mathrm{d}\eta\nonumber\\
		\leq & \int_{0}^{\infty} p\eta^{p-1}\mathbb{P}\left(\int_{0}^{T}\Vert \sigma(s,\cdot)\Vert^q_{L^2_\lambda}\mathrm{d}s>\eta^q\right)\mathrm{d}\eta\nonumber\\
		&\quad+C_{\lambda,T,q}\int_{0}^{\infty} p\eta^{p-1-q}~\mathbb{E}\min\left\{\eta^q,\int_{0}^{T}\Vert\sigma(s,\cdot)\Vert^q_{L^2_\lambda}\mathrm{d}s\right\}\mathrm{d}\eta \nonumber\\
		\leq& C_{\lambda, T,p, q} ~\mathbb{E}\left(\int_{0}^{t}\Vert\sigma(s,\cdot)\Vert_{L^2_{\lambda}}^q \mathrm{d} s\right)^{\frac{p}{q}} \nonumber\\
		\leq& C_{\lambda, T,p,q}~\mathbb{E}\left[\sup_{s \leq T}\Vert\sigma(s,\cdot)\Vert_{L^2_{\lambda}}^{\frac{(q-p) p}{q}}\left(\int_{0}^{T}\Vert\sigma(s,\cdot)\Vert_{L^2_{\lambda}}^p \mathrm{d} s\right)^{\frac{p}{q}}\right] \nonumber\\
		\leq& \epsilon~\mathbb{E} \sup_{s\leq T}\Vert\sigma(s,\cdot)\Vert_{L^2_{\lambda}}^p+C_{\epsilon,\lambda,T,p}~\mathbb{E} \int_{0}^{T}\Vert\sigma(s,\cdot)\Vert_{L^2_{\lambda}}^p \mathrm{d}s,
	\end{align*}
	where $C_{\lambda,T,p,q}=1+\frac{q}{q-p}C_{\lambda,T,q}$, and  Young's inequality is used in the last step.
\end{proof}

\section{Proof of Theorem \ref{Lresult}}
$E=L^2_{tem}$. Take $\nu\ll \mu$ on $C([0, T],E)$.
Define the corresponding measure $\mathbb{Q}$ by (\ref{add 0303.1}).
Let $h(t,x)$ be the corresponding random field appeared in Lemma \ref{CtemGirsa}. Then  the solution $u(t,x)$ of equation (\ref{3.1}) satisfies the following SPDE under the measure $\mathbb{Q}$,
\begin{align}
	\label{111.3}
	u(t,x)= & P_t u_0(x) + \int_{0}^{t}\int_{\mathbb{R}} p_{t-s}(x,y)b(u(s,y))\,\mathrm{d}y\mathrm{d}s \nonumber\\
	& +\int_{0}^{t}\int_{\mathbb{R}} p_{t-s}(x,y)\sigma(u(s,y))\,\widetilde{W}(\mathrm{d}s,\mathrm{d}y)  \nonumber\\
	& + \int_{0}^{t}\int_{\mathbb{R}} p_{t-s}(x,y)\sigma(u(s,y))h(s,y)\,\mathrm{d}y\mathrm{d}s.
\end{align}
Consider the solution of the following SPDE:
\begin{align}
	\label{111.2}
	v(t,x)= & P_t u_0(x) + \int_{0}^{t}\int_{\mathbb{R}} p_{t-s}(x,y)b(v(s,y))\,\mathrm{d}y\mathrm{d}s \nonumber\\
	& +\int_{0}^{t}\int_{\mathbb{R}} p_{t-s}(x,y)\sigma(v(s,y))\,\widetilde{W}(\mathrm{d}s,\mathrm{d}y).
\end{align}
By Lemma \ref{CtemGirsa} it follows that under the measure $\mathbb{Q}$, the law of $(v,u)$ forms a coupling of $(\mu, \nu)$.
Therefore by the definition of the Wasserstein distance,
\begin{align*}
	W_2(\nu, \mu)^2 & \leq  \mathbb{E}^{\mathbb{Q}} \left[\sup_{t\in[0,T]}d(u(t), v(t))^2 \right].
\end{align*}
Here $d$ stands for the metric of space $E$. More precisely, we have
\begin{align*}
	W_2(\nu, \mu)^2 & \leq  \mathbb{E}^{\mathbb{Q}}\left[\sup_{t\in[0,T]}\bigg| \sum_{n=1}^{\infty}\frac{1}{2^n}\min\Big\{1, \|u(t)-v(t)\|_{L^2_{1/n}}\Big\} \bigg|^2\right] \nonumber\\
	&\leq \mathbb{E}^{\mathbb{Q}}\left[\sup_{t\in[0,T]} \sum_{n=1}^{\infty}\frac{1}{2^n}\|v(t)-u(t)\|^2_{L^2_{1/n}} \right] \nonumber\\
	& \leq\sum_{n=1}^{\infty}\frac{1}{2^n} \mathbb{E}^{\mathbb{Q}}\left[\sup_{t\in[0,T]}\left( \int_{\mathbb{R}} |v(t,x)-u(t,x)|^2 e^{-\frac{2}{n} |x|}\mathrm{d}x\right)\right].
\end{align*}
In view of (\ref{4.4}), to prove the quadratic transportation cost inequality
\begin{align}
	W_2(\nu, \mu)\leq \sqrt{2C H(\nu|\mu)},
\end{align}
it is sufficient to show that exists a constant $C$ independent of $n$ such that for any $n\in\mathbb{N}$,
\begin{align}\label{Lineq}
	\mathbb{E}^{\mathbb{Q}}\left[\sup_{t\in[0,T]}\left(\int_{\mathbb{R}} |u(t,x)-v(t,x)|^2 e^{-\frac{2}{n} |x|}\mathrm{d}x\right)\right]
	\leq C \mathbb{E}^{\mathbb{Q}}\left[\int_{0}^{T}\int_{\mathbb{R}} h^2(s,y)\,\mathrm{d}y\mathrm{d}s\right]
\end{align}
when the right-hand side of (\ref{Lineq}) is finite.
To this end, for $\lambda >0$ let us estimate
\begin{align*}
	\mathbb{E}^{\mathbb{Q}}\left[\sup_{t\in[0,T]}\left( \int_{\mathbb{R}} |u(t,x)-v(t,x)|^2 e^{-2\lambda |x|}\mathrm{d}x\right)\right].
\end{align*}
For convenience, in the sequel we denote $\mathbb{E}^{\mathbb{Q}}$ still  by the symbol $\mathbb{E}$. From (\ref{111.3}) and (\ref{111.2}) it follows that
\begin{align}\label{L0}
	&\mathbb{E} ~\sup_{t \leq T}\left\{ \int_{\mathbb{R}}|u(t, x)-v(t, x)|^2 e^{-2 \lambda|x|} \mathrm{d}x \right\}\nonumber\\
	\leq& 3~\mathbb{E} ~\sup_{t \leq T} \left\{\int_{\mathbb{R}}\left|\int_{0}^{t} \int_{\mathbb{R}} p_{t-s}(x, y)[b(u(s, y))-b(v(s, y))] \mathrm{d}s \mathrm{d}y\right|^2 e^{-2 \lambda|x|} \mathrm{d}x\right\} \nonumber\\
	&+3~\mathbb{E}~ \sup_{t \leq T} \left\{\int_{\mathbb{R}}\left|\int_{0}^{t} \int_{\mathbb{R}} p_{t-s}(x, y)[\sigma(u(s, y))-\sigma(v(s, y))] \widetilde{W}(\mathrm{d}s, \mathrm{d}y)\right|^2 e^{-2 \lambda|x|} \mathrm{d}x\right\}\nonumber\\
	&+3~\mathbb{E} ~\sup_{t \leq T}\left\{ \int_{\mathbb{R}}\left|\int_{0}^{t} \int_{\mathbb{R}} p_{t-s}(x, y) \sigma(u(s, y)) h(s, y) \mathrm{d}s \mathrm{d}y\right|^2 e^{-2 \lambda|x|} \mathrm{d}x\right\} \nonumber\\
	=&:3(I_1+I_2+I_3).
\end{align}
By H\"{o}lder's inequality and (\textbf{H1}) we have that
\begin{align}\label{L1}
	I_1  \leq & \mathbb{E}~\sup_{t \leq T}\left\{\int_{\mathbb{R}} \left(\int_{0}^{t} \int_{\mathbb{R}} p_{t-s}(x, y) \mathrm{d}y \mathrm{d}s\right)\right.\nonumber\\
	&\quad\times
	\left.\left(\int_{0}^{t} \int_{\mathbb{R}} p_{t-s}(x, y)|b(u(s, y))-b(v(s, y))|^2 \mathrm{d}y \mathrm{d}s\right) \,e^{-2 \lambda|x|} \mathrm{d}x\right\} \nonumber\\
	\leq & T L^2_b~\mathbb{E}~\sup_{t \leq T}\left\{\int_{\mathbb{R}} \int_{0}^{t} \int_{\mathbb{R}} p_{t-s}(x, y)|u(s, y)-v(s, y)|^2 \mathrm{d}y\mathrm{d}s\,e^{-2 \lambda|x|} \mathrm{d}x\right\} \nonumber\\
	\leq & 2Te^{2\lambda^2T} L^2_b~\mathbb{E}\int_{0}^{T}\int_{\mathbb{R}}|u(s,y)-v(s,y)|^2 e^{-2 \lambda|y|} \mathrm{d}y\mathrm{d}s \nonumber\\
	\leq&2Te^{2\lambda^2T} L^2_b~\int_{0}^{T}\mathbb{E} \sup_{r \leq s}\left\{\int_{\mathbb{R}}|u(r,y)-v(r,y)|^2 e^{-2 \lambda|y|} \mathrm{d}y\right\}\mathrm{d}s.
\end{align}
In the third inequality, Fubini's theorem and \eqref{103.1} were used. According to Lemma \ref{moment2}, for  $\epsilon>0$,
\begin{align}\label{L2}
	&\mathbb{E} \sup _{t \leq T} \left\{\int_{\mathbb{R}}\left|\int_{0}^{t} \int_{\mathbb{R}} p_{t-s}(x, y)[\sigma(u(s,y))-\sigma(v(s,y))]\widetilde{W}(\mathrm{d}s,\mathrm{d}y)\right|^2 e^{-2 \lambda|x|} \mathrm{d}x\right\} \nonumber\\
	\leq& \epsilon ~\mathbb{E} ~\sup_{s \leq T}\left\{ \int_{\mathbb{R}}|\sigma(u(s, y))-\sigma(v(s, y))|^2e^{-2\lambda|y|} \mathrm{d}y\right\} \nonumber\\
	&+C_{\epsilon, \lambda,T} \int_{0}^{T} \mathbb{E} \int_{\mathbb{R}}|\sigma(u(s, y))-\sigma(v(s, y))|^2 e^{-2\lambda|y|} \mathrm{d}y\, \mathrm{d}s \nonumber\\
	\leq& \epsilon L^2_\sigma ~\mathbb{E}~ \sup_{s \leq T} \left\{\int_{\mathbb{R}}|u(s, y)-v(s, y)|^2e^{-2\lambda|y|} \mathrm{d}y \right\}\nonumber\\
	&+C_{\epsilon, \lambda, T}  L^2_\sigma \int_{0}^{T} \mathbb{E} \sup_{r \leq s}\left\{\int_{\mathbb{R}}|u(r, y)-v(r, y)|^2 e^{-2\lambda|y|} \mathrm{d}y\right\} \mathrm{d}s.
\end{align}
For the third term $I_3$, by the boundedness of $\sigma$(see (\textbf{H2})), Minkowski's inequality for integrals, \eqref{kernal} and H\"{o}lder's inequality we have
\begin{align}\label{L3}
	I_3 \leq & K_\sigma^2~\mathbb{E}~\sup_{t \leq T} \left\Vert \int_{0}^{t}P_{t-s}h(s,\cdot)\mathrm{d}s\right\Vert^2_{L^2_\lambda}\nonumber\\
	\leq & K_\sigma^2~\mathbb{E}~\sup_{t \leq T}  \left|\int_{0}^{t}\Vert P_{t-s}h(s,\cdot)\Vert_{L^2_\lambda}\mathrm{d}s\right|^2\nonumber\\
	\leq & K_\sigma^2~\mathbb{E}~\sup_{t \leq T}\left| \int_{0}^{t}\sqrt{2}e^{\lambda^2t}\Vert h(s,\cdot)\Vert_{L^2_\lambda}\mathrm{d}s\right|^2 \nonumber\\	
	\leq &2e^{2\lambda^2T} T K_\sigma^2~\mathbb{E}~ \int_{0}^{T} \Vert h(s,\cdot)\Vert^2_{L^2_\lambda} \mathrm{d}s \nonumber\\
	\leq & 2e^{2\lambda^2T} T K_\sigma^2~\mathbb{E} \int_{0}^{T} \int_{\mathbb{R}} |h(t,x)|^2\mathrm{d}x\mathrm{d}t.
\end{align}
Substituting \eqref{L1}-\eqref{L3} into \eqref{L0}, we have
\begin{align*}
	&(1-3\epsilon L^2_\sigma)\mathbb{E} ~\sup_{t \leq T}\left\{ \int_{\mathbb{R}}|u(t, x)-v(t, x)|^2 e^{-2 \lambda|x|} \mathrm{d}x \right\}\nonumber\\
	\leq & C_{\epsilon,\lambda,T,L_b,L_\sigma}\int_{0}^{T}\mathbb{E} \sup_{r \leq s}\left\{\int_{\mathbb{R}}|u(s,y)-v(s,y)|^2 e^{-2 \lambda|y|} \mathrm{d}y\right\}\mathrm{d}s\nonumber\\
	&+6Te^{2\lambda^2T} K_\sigma^2~\mathbb{E} \int_{0}^{T} \int_{\mathbb{R}} |h(t,x)|^2\mathrm{d}x\mathrm{d}t.
\end{align*}
Taking any $0<\epsilon<\frac{1}{3L^2_\sigma}$ and applying the Gronwall's inequality we obtain
\begin{align*}
	&\mathbb{E} \sup_{t \leq T} \left\{\int_{\mathbb{R}}|u(t, x)-v(t, x)|^2 e^{-2\lambda|x|} \mathrm{d}x \right\} \nonumber\\
	\leq & \frac{6Te^{2\lambda^2T} K_\sigma^2}{1-3\epsilon L^2_\sigma}\exp\left\{\frac{TC_{\epsilon,\lambda,T,L_b,L_\sigma}}{1-3\epsilon L^2_\sigma}\right\} ~\mathbb{E} \int_{0}^{T} \int_{\mathbb{R}}|h(t,x)|^2 \mathrm{d}x\mathrm{d}t.
\end{align*}
Taking infimum over $\epsilon$, we get
\begin{align*}
	\mathbb{E} \sup_{t\leq T} \left\{\int_{\mathbb{R}}|u(t, x)-v(t, x)|^2 e^{-2 \lambda|x|} \mathrm{d}x\right\}
	\leq 
	C_{\lambda,T,L_b,L_\sigma,K_\sigma}~\mathbb{E} \int_{0}^{T} \int_{\mathbb{R}}|h(t,x)|^2 \mathrm{d}x\mathrm{d}t.
\end{align*}
Noting that $C_{\lambda,T,L_b,L_\sigma,K_\sigma}$ is increasing with respect to $\lambda$, it follows that  for each $n\in\mathbb{N}$,
\begin{align*}
	\mathbb{E}\sup_{t\leq T}\left(\int_{\mathbb{R}}|u(t,x)-v(t,x)|^2 e^{-\frac{2}{n} |x|}\mathrm{d}x\right)
	\leq  
	C_{1,T,L_b,L_\sigma,K_\sigma}~\mathbb{E} \int_{0}^{T} \int_{\mathbb{R}}|h(t,x)|^2 \mathrm{d}x\mathrm{d}t,
\end{align*}
This proves (\ref{Lineq}) and hence the proof of Theorem \ref{Lresult} is complete.

\section{Proof of Theorem \ref{Ctemresult}}

In the proof of the main results, the following moments estimates for stochastic convolution against space-time white noise obtained in \cite{SZ2} will play a crucial role.

\begin{lemma}\label{210121.1036}
	Let $h: \mathbb{R}_{+} \longmapsto \mathbb{R}_{+}$ be an increasing function.
	Let $\{\sigma(s,y): (s,y)\in\mathbb{R}_+\times \mathbb{R}\}$ be an adapted random field such that the following stochastic convolution with respect to space-time white noise is well defined.
	Let $\tau$ be a stopping time.
	
	\begin{itemize}
		\item [(i)] Then  for any $p>10$ and $T>0$, there exists a constant $\Theta_{p, h(T), T}>0$ such that
		\begin{align}\label{210119.2121}
			& \mathbb{E} \sup_{(t,x)\in [0,T\wedge\tau]\times \mathbb{R}}\left\{\left|\int_{0}^{t}\int_{\mathbb{R}}p_{t-s}(x,y)\sigma(s,y)W(ds,dy)\right| e^{-h(t)|x|}\right\}^p \nonumber\\
			\leq & \Theta_{p, h(T), T} ~\mathbb{E}\int_0^{T\wedge\tau} \int_{\mathbb{R}}|\sigma(t,x)|^p e^{-p h(t) |x|} \, \mathrm{d}x \mathrm{d}t .
		\end{align}
		
		\item [(ii)] Then for any $\epsilon, T>0$ and $0<p\leq 10$, there exists a constant $\Theta_{\epsilon, p, h(T), T}>0$ such that
		\begin{align}\label{220525.1627}
			& \mathbb{E} \sup_{(t,x)\in [0,T\wedge\tau]\times \mathbb{R}}\left\{\left|\int_{0}^{t}\int_{\mathbb{R}}p_{t-s}(x,y)\sigma(s,y)W(ds,dy)\right|  e^{-h(t)|x|}\right\}^p \nonumber\\
			\leq & \epsilon  ~\mathbb{E} \sup_{(t,x)\in[0,T\wedge\tau]\times\mathbb{R}}\left(|\sigma(t,x)| e^{-h(t)|x|} \right)^p  \nonumber\\
			& + \Theta_{\epsilon,p,h(T),T} ~\mathbb{E}\int_0^{T\wedge\tau}\int_{\mathbb{R}} \left|\sigma(t,x)\right|^p e^{-ph(t)|x|}\,\mathrm{d}x\mathrm{d}t.
		\end{align}
	\end{itemize}
Here $\Theta_{p, h(T), T}$ and $\Theta_{\epsilon, p, h(T), T}$ are increasing w.r.t. $h(T)$ and w.r.t. $T$.
\end{lemma}

\vskip 0.5cm
\noindent {\bf Proof of Theorem \ref{Ctemresult}}.
\vskip 0.3cm

Let $E=C_{tem}$. As in the proof Theorem \ref{Lresult}, take $\nu\ll \mu$ on $C([0, T],E)$.
Define the corresponding measure $\mathbb{Q}$ by (\ref{add 0303.1}).
Let $h(t,x)$ be the corresponding random field appeared in Lemma \ref{CtemGirsa}. Then  the solution $u(t,x)$ of equation (\ref{3.1}) satisfies the following SPDE under the measure $\mathbb{Q}$,
\begin{align}
	\label{5-1}
	u(t,x)= & P_t u_0(x) + \int_{0}^{t}\int_{\mathbb{R}} p_{t-s}(x,y)b(u(s,y))\,\mathrm{d}y\mathrm{d}s \nonumber\\
	& +\int_{0}^{t}\int_{\mathbb{R}} p_{t-s}(x,y)\sigma(u(s,y))\,\widetilde{W}(\mathrm{d}s,\mathrm{d}y)  \nonumber\\
	& + \int_{0}^{t}\int_{\mathbb{R}} p_{t-s}(x,y)\sigma(u(s,y))h(s,y)\,\mathrm{d}y\mathrm{d}s.
\end{align}
Consider the solution of the following SPDE:
\begin{align}
	\label{5-2}
	v(t,x)= & P_t u_0(x) + \int_{0}^{t}\int_{\mathbb{R}} p_{t-s}(x,y)b(v(s,y))\,\mathrm{d}y\mathrm{d}s \nonumber\\
	& +\int_{0}^{t}\int_{\mathbb{R}} p_{t-s}(x,y)\sigma(v(s,y))\,\widetilde{W}(\mathrm{d}s,\mathrm{d}y).
\end{align}
By Lemma \ref{CtemGirsa} it follows that under the measure $\mathbb{Q}$, the law of $(v,u)$ forms a coupling of $(\mu, \nu)$.
Therefore by the definition of the Wasserstein distance,
\begin{align*}
		W_2(\nu, \mu)^2 & \leq  \mathbb{E}^{\mathbb{Q}} \left[\sup_{t\in[0,T]}d(u(t), v(t))^2 \right].
\end{align*}
Here $d$ stands for the metric of space $E$. More precisely, we have
\begin{align*}
	W_2(\nu, \mu)^2 & \leq \mathbb{E}^{\mathbb{Q}}\left[\sup_{t\in[0,T]}\left| \sum_{n=1}^{\infty}\frac{1}{2^n}\min\left\{1, \varrho_{1/n}(u(t),v(t))\right\} \right|^2\right] \nonumber\\
	&\leq \mathbb{E}^{\mathbb{Q}}\left[\sup_{t\in[0,T]} \sum_{n=1}^{\infty}\frac{1}{2^n} (\varrho_{1/n}(u(t),v(t)))^2 \right] \nonumber\\
	& \leq\sum_{n=1}^{\infty}\frac{1}{2^n} \mathbb{E}^{\mathbb{Q}}\left[\sup_{(t,x)\in[0,T]\times\mathbb{R}} \Big( |u(t,x)-v(t,x)|^2 e^{-\frac{2}{n} |x|}\Big) \right]
\end{align*}
Similarly, to prove the quadratic transportation cost inequality,
it is sufficient to show that exists a constant $\Theta$ independent of $n$ such that for any $n\in\mathbb{N}$,
\begin{align}\label{Ctemineq}
	\mathbb{E}^{\mathbb{Q}}\left[\sup_{(t,x)\in[0,T]\times\mathbb{R}}\Big( |u(t,x)-v(t,x)|^2 e^{-\frac{2}{n} |x|}\Big)\right]
	\leq \Theta \mathbb{E}^{\mathbb{Q}}\left[\int_{0}^{T}\int_{\mathbb{R}} h^2(s,y)\,\mathrm{d}y\mathrm{d}s\right]
\end{align}
when the right-hand side of (\ref{Ctemineq}) is finite. To this end, for $\lambda >0$, we consider
\begin{align*}
	Y(T):=\mathbb{E}^{\mathbb{Q}}\left[\sup_{(t,x)\in[0,T]\times\mathbb{R}}\left( |u(t,x)-v(t,x)|^2 e^{-2\lambda |x|}\right)\right].
\end{align*}
Again we still denote $\mathbb{E}^{\mathbb{Q}}$ by the symbol $\mathbb{E}$ for convenience. From (\ref{5-1}) and (\ref{5-2}) it follows that
\begin{align}\label{add 0302.1}
	Y(T)\leq&
	3\mathbb{E}\sup_{(t,x)\in[0,T]\times\mathbb{R}}\left\{\left|\int_{0}^{t}\int_{\mathbb{R}} p_{t-s}(x,y)\big[b(v(s,y))-b(u(s,y))\big]\,\mathrm{d}y\mathrm{d}s \right|^2e^{-2\lambda |x|}\right\}  \nonumber\\
	&+3\mathbb{E}\sup_{(t,x)\in[0,T]\times\mathbb{R}}\left\{\left|\int_{0}^{t}\int_{\mathbb{R}} p_{t-s}(x,y)\big[\sigma(v(s,y))-\sigma(u(s,y))\big]\,\widetilde{W}(\mathrm{d}s,\mathrm{d}y)\right|^2e^{-2\lambda |x|}\right\} \nonumber\\
	&+3\mathbb{E}\sup_{(t,x)\in[0,T]\times\mathbb{R}}\left\{\left|\int_{0}^{t}\int_{\mathbb{R}}p_{t-s}(x,y)\sigma(u(s,y))h(s,y)\mathrm{d}y\mathrm{d}s \right|^2e^{-2\lambda |x|}\right\} \nonumber \\
	=& :3(I + II +III).
\end{align}

\noindent By the assumption ({\bf H1}) and (\ref{103.1}), the term $I$ can be estimated as follows:
{\allowdisplaybreaks\begin{align}\label{term I}
		I \leq & L_b^2
		~\mathbb{E}\sup_{(t,x)\in[0,T]\times\mathbb{R}}\left\{\left|\int_{0}^{t}\int_{\mathbb{R}} p_{t-s}(x,y)|u(s,y)-v(s,y)|\mathrm{d}y\mathrm{d}s\right|^2e^{-2
			\lambda |x|}\right\}  \nonumber\\
		\leq & L_b^2 ~\mathbb{E}\sup_{(t,x)\in[0,T]\times\mathbb{R}}\left\{  \int_{0}^{t} \sup_{y\in\mathbb{R}} \left(|u(s,y)-v(s,y)|^2e^{-2\lambda |y|}\right) \int_{\mathbb{R}} p_{t-s}(x,y)e^{2\lambda |y|}\mathrm{d}y\mathrm{d}s~e^{-2\lambda|x|}  \right\}\nonumber\\
		\leq & 2e^{2\lambda^2T} L_b^2 ~\mathbb{E}\int_{0}^{t}\sup_{(r,y)\in[0,s]\times\mathbb{R}}\left(|u(r,y)-v(r,y)|^2e^{-2\lambda|y|}\right)\mathrm{d}s  \nonumber\\
		= &  2e^{2\lambda^2T} L_b^2 \int_{0}^{t} Y(s)\mathrm{d}s .
	\end{align}
	
	\noindent By the boundedness of $\sigma$ and using Hölder's inequality and (\ref{103.2}) the term $III$ can be bounded as follows,
	\begin{align}\label{term III}
		III \leq & K_{\sigma}^2~\mathbb{E}\sup_{(t,x)\in[0,T]\times\mathbb{R}}\left(\int_{0}^{t}\int_{\mathbb{R}} p_{t-s}^2(x,y) \mathrm{d}y\mathrm{d}s \cdot\int_{0}^{t}\int_{\mathbb{R}} |h(s,y)|^2 \mathrm{d}y\mathrm{d}s \right) \nonumber\\
		\leq & \frac{1}{\sqrt{\pi}}  K_{\sigma}^2 ~\mathbb{E} \sup_{t\in[0,T]}\left( \int_{0}^{t} \frac{1}{\sqrt{t-s}}\mathrm{d}s \int_{0}^{t}\int_{\mathbb{R}} |h(s,y)|^2 \mathrm{d}y\mathrm{d}s\right) \nonumber\\
		\leq & \frac{2\sqrt{T}}{\sqrt{\pi}}  K_{\sigma}^2  ~\mathbb{E}\int_{0}^{T}\int_{\mathbb{R}} |h(s,y)|^2 \mathrm{d}y\mathrm{d}s.
	\end{align}
	
	\noindent For the term $II$, by the assumption ({\bf H2}) and the estimate (\ref{220525.1627}) we obtain that for any $\epsilon>0$,
	\begin{align}\label{term II}
		II \leq & \epsilon L_{\sigma}^2 ~\mathbb{E}\sup_{(t,x)\in[0,T]\times\mathbb{R}}\left(|u(t,x)-v(t,x)|^2e^{-2\lambda |x|}\right) \nonumber\\
		& + \Theta_{\epsilon,\lambda,T} L_{\sigma}^2 ~\mathbb{E} \int_{0}^{T}\int_{\mathbb{R}} \left(\left|u(t,x)-v(t,x)\right|^2e^{-2\lambda |x|}\right)\mathrm{d}x\mathrm{d}t.
	\end{align}
	It remains to give an estimate for the following integral.
	\begin{align*}
		\mathbb{E}\int_{0}^{T}\int_{\mathbb{R}} \left(\left|u(t,x)-v(t,x)\right|^2e^{-2\lambda |x|}\right)\mathrm{d}x\mathrm{d}t.
	\end{align*}
	
	\begin{lemma}\label{L2}
		For any $\lambda>0$, $0\leq t\leq T$, there exists a constant $\Theta_{\lambda,T,L_b,L_\sigma,K_\sigma}$ increasing w.r.t. $\lambda$ such that
		\begin{align}\label{F}
			\mathbb{E}\int_{0}^{T}\int_{\mathbb{R}} \left|u(t,x)-v(t,x)\right|^2e^{-2\lambda |x|}\,\mathrm{d}x\mathrm{d}t\leq \Theta_{\lambda,T,L_b,L_\sigma,K_\sigma} 	\mathbb{E}\int_{0}^{T}\int_{\mathbb{R}}|h(t,x)|^2\,\mathrm{d}x\mathrm{d}t.
		\end{align}
	\end{lemma}
	
	\begin{proof}
		Define
		\begin{align*}
			F(t):=\mathbb{E}\int_{\mathbb{R}}\left(\left|u(t,x)-v(t,x)\right|^2e^{-2\lambda |x|}\right)\mathrm{d}x,\ t\in[0,T].
		\end{align*}
		By (\ref{5-1}) and (\ref{5-2}), we have
		\begin{align}\label{F0}
			F(t)\leq & 3\mathbb{E}\int_{\mathbb{R}} \left|\int_{0}^{t}\int_{\mathbb{R}}p_{t-s}(x,y)\left[b(u(s,y))-b(v(s,y))\right]\mathrm{d}y\mathrm{d}s\right|^2e^{-2\lambda |x|}\mathrm{d}x \nonumber\\
			&+3\mathbb{E}\int_{\mathbb{R}} \left|\int_{0}^{t}\int_{\mathbb{R}}p_{t-s}(x,y)\sigma(u(s,y))h(s,y)\mathrm{d}y\mathrm{d}s\right|^2e^{-2\lambda |x|}\mathrm{d}x \nonumber\\
			&+3\mathbb{E}\int_{\mathbb{R}} \left|\int_{0}^{t}\int_{\mathbb{R}}p_{t-s}(x,y)\left[\sigma(u(s,y))-\sigma(v(s,y))\right]\tilde{W}(\mathrm{d}s,\mathrm{d}y)\right|^2e^{-2\lambda |x|}\mathrm{d}x \nonumber\\
			=&: 3J_1+3J_2+3J_3.
		\end{align}
		
		For the term $J_1$, using the assumption ({\bf H1}), Hölder's inequality and (\ref{103.1}) we obtain
		\begin{align}\label{F1}
			J_1\leq& L_b^2 ~\mathbb{E}\int_{\mathbb{R}}\left(\int_{0}^{t}\int_{\mathbb{R}}p_{t-s}(x,y)\mathrm{d}y\mathrm{d}s\cdot\int_{0}^{t}\int_{\mathbb{R}} p_{t-s}(x,y)|u(s,y)-v(s,y)|^2\mathrm{d}y\mathrm{d}s ~e^{-2\lambda |x|}\right)\mathrm{d}x \nonumber\\
			\leq & 2te^{2\lambda^2 t} L_b^2 ~\mathbb{E}\int_{0}^{t} \int_{\mathbb{R}} |u(s,y)-v(s,y)|^2e^{-2\lambda |y|}\mathrm{d}y~\mathrm{d}s \nonumber\\
			=& 2te^{2\lambda^2 t }L_b^2\int_{0}^{t} F(s)\mathrm{d}s,
		\end{align}
		where we used the Fubini theorem in the last step.
		
		And also we can estimate $J_2$ as follows,
		\begin{align}\label{F2}
			J_2\leq & K_\sigma^2~\mathbb{E}\int_{\mathbb{R}} \left(\int_{0}^{t} \int_{\mathbb{R}}p_{t-s}(x,y)\mathrm{d}y\mathrm{d}s\cdot\int_{0}^{t}\int_{\mathbb{R}} p_{t-s}(x,y)|h(s,y)|^2\mathrm{d}y\mathrm{d}s ~e^{-2\lambda |x|}\right)\mathrm{d}x \nonumber\\
			\leq & tK_\sigma^2 ~\mathbb{E}\int_{\mathbb{R}}\int_{0}^{t}\int_{\mathbb{R}} p_{t-s}(x,y)|h(s,y)|^2\mathrm{d}y\mathrm{d}s e^{-2\lambda |x|} \mathrm{d}x \nonumber\\
			\leq & 2te^{2\lambda^2 t} K_\sigma^2 ~\mathbb{E}\int_{0}^{t}\int_{\mathbb{R}}|h(s,y)|^2\mathrm{d}y\mathrm{d}s,
		\end{align}
		where the boundedness of $\sigma$ and (\ref{103.1}) were used.
		
		As for the term $J_3$, by the Fubini theorem(see \cite{Wa}) and the Bukrholder-Davis-Gundy's inequality (see \cite{K}), then using the assumption ({\bf H2}) and (\ref{103.2}) we have
		\begin{align}\label{F3}
			J_3\leq & ~8\int_{\mathbb{R}} \mathbb{E}\int_{0}^{t}\int_{\mathbb{R}} p_{t-s}(x,y)^2 |\sigma(u(s,y))-\sigma(v(s,y))|^2\mathrm{d}y\mathrm{d}s~e^{-2\lambda |x|}\mathrm{d}x \nonumber\\
			\leq & 8e^{\lambda^2 t}L_\sigma^2 ~\mathbb{E}\int_{0}^{t}\int_{\mathbb{R}} \frac{1}{\sqrt{\pi(t-s)}} |u(s,y)-v(s,y)|^2e^{-2\lambda|y|}\mathrm{d}y\mathrm{d}s \nonumber\\
			= & \frac{8}{\sqrt{\pi}} e^{\lambda^2 t}L_\sigma^2 \int_{0}^{t} \frac{F(s)}{\sqrt{t-s}} \mathrm{d}s,
		\end{align}
		where in the last step we used the Fubini theorem.
		
		Substituting (\ref{F1})-(\ref{F3}) into (\ref{F0}) leads to
		\begin{align}\label{F4}
			F(t)\leq & 6te^{2\lambda^2 t} K_\sigma^2 ~\mathbb{E}\int_{0}^{t}\int_{\mathbb{R}}|h(s,y)|^2\mathrm{d}y\mathrm{d}s+ 6te^{2\lambda^2 t }L_b^2\int_{0}^{t} F(s)\mathrm{d}s \nonumber\\
			&+\frac{24}{\sqrt{\pi}} e^{\lambda^2 t}L_\sigma^2 \int_{0}^{t} \frac{F(s)}{\sqrt{t-s}} \mathrm{d}s.
		\end{align}
		Then iterating \eqref{F4} and using Gronwall's inequality yield that for any $t\in[0,T]$,
		\begin{align*}
			\mathbb{E}\int_{\mathbb{R}} |u(t,x)-v(t,x)|^2e^{-2\lambda |x|} \mathrm{d}x
			\leq \Theta_{\lambda,T,L_b,L_\sigma,K_\sigma}  \mathbb{E}\int_{0}^{T}\int_{\mathbb{R}}|h(s,y)|^2\mathrm{d}y\mathrm{d}s.
		\end{align*}
		Then \eqref{F} follows from the preceding inequality.
	\end{proof}

	\vskip 0.5cm
	Now we come back to estimate $II$. Combining (\ref{term II}) with (\ref{F}) leads to
	\begin{align}\label{term II2}
		II\leq \epsilon L_\sigma^2 Y(T)+  L_\sigma^2 \Theta_{\epsilon, \lambda, T}\,\Theta_{\lambda,T,L_b,L_\sigma,K_\sigma}\,\mathbb{E}\int_{0}^{T}\int_{\mathbb{R}}|h(s,y)|^2\mathrm{d}y\mathrm{d}s.
	\end{align}
	Then putting (\ref{add 0302.1}), (\ref{term I}), (\ref{term III}) and  (\ref{term II2}) together, we obtain
	\begin{align}\label{4.1}
		Y(T)\leq  &6e^{2\lambda^2T} L_b^2 \int_{0}^{t} Y(s)\mathrm{d}s+3\epsilon L_\sigma^2 Y(T) \nonumber\\
		&+ \Theta_{\epsilon,\lambda, T,L_b,L_\sigma, K_\sigma}~\mathbb{E}\int_{0}^{T}\int_{\mathbb{R}}|h(t,x)|^2\mathrm{d}x\mathrm{d}t.
	\end{align}
Note that $\Theta_{\epsilon,\lambda, T,L_b,L_\sigma, K_\sigma}$ is increasing with respect to $\lambda$ and with repsect to $T$.
	
	Recall that for any $\lambda >0$, (see \cite{S})
	\begin{align*}
		& \mathbb{E}\sup_{(t,x)\in[0,T]\times\mathbb{R}}
		\left(|u(t,x)|^2e^{-2\lambda |x|}\right)< \infty , \\
		& \mathbb{E}\sup_{(t,x)\in[0,T]\times\mathbb{R}}\left(|v(t,x)|^2e^{-2\lambda |x|}\right)< \infty .
	\end{align*}
	Hence $Y(T)<\infty$ for any $T>0$. Clearly, (\ref{4.1}) still holds if we replace $T$ with any $t\in[0,T]$. Next, taking any $0<\epsilon<\frac{1}{3 L_{\sigma}^2}$ and applying the Gronwall's inequality we obtain
	\begin{align*}
		Y(T)\leq \frac{\Theta_{\epsilon,\lambda, T,L_b,L_\sigma, K_\sigma}}{1-3\epsilon L_\sigma^2} \exp\left\{\frac{6Te^{2\lambda^2T} L_b^2 }{1-3\epsilon L_\sigma^2}\right\} \mathbb{E}\int_{0}^{T}\int_{\mathbb{R}}|h(t,x)|^2\mathrm{d}x\mathrm{d}t .
	\end{align*}
	Since $0<\epsilon<\frac{1}{3L_\sigma^2}$ is arbitrary, we get
	\begin{align*}
		Y(T)\leq 
		\Theta_{\lambda,T,L_b,L_\sigma,K_\sigma}~\mathbb{E}\int_{0}^{T}\int_{\mathbb{R}}|h(t,x)|^2\mathrm{d}x\mathrm{d}t.
	\end{align*}
	By the definition of $Y(T)$, taking $\lambda=\frac{1}{n},\ n\in\mathbb{N}$ in the above inequality, and noting that $\Theta_{\lambda, T,L_b,L_\sigma,K_\sigma}$ is increasing with respect to $\lambda$, we have for each $n\in\mathbb{N}$,
	\begin{align*}
		\mathbb{E}\sup_{(t,x)\in[0,T]\times\mathbb{R}}\left( |u(t,x)-v(t,x)|^2 e^{-\frac{2}{n} |x|}\right)
		\leq  
		\Theta_{\lambda,T,L_b,L_\sigma,K_\sigma}~\mathbb{E}\int_{0}^{T}\int_{\mathbb{R}}|h(t,x)|^2\mathrm{d}x\mathrm{d}t.
	\end{align*}
	This proves (\ref{Ctemineq}) and hence the proof of Theorem \ref{Ctemresult} is complete.

	\vskip 0.5cm
	
	\section{Proof of Corollary \ref{Lrandom} and \ref{LCtemrandom}}
	
	\quad In this section we generalize the transportation cost inequalities  from deterministic initial values to random initial values. We will follow the same approach as in the proof of Theorem 3.1 in \cite{WZ}.
	
	\vskip 0.5cm
	\noindent \emph{Proof of Corollary \ref{Lrandom}}
	
	In fact, the result of Theorem \ref{Lresult} can be written in the form:
	\begin{align}
		W_2(Q,P^{u_0})^2\leq c_1 H(Q|P^{u_0}),\ Q\in\mathcal{P}(C([0,T],L^2_{tem})),\ u_0\in L^2_{tem}.
	\end{align}
	According to Theorem 2.1 in \cite{WZ}, it suffices to show that the Warssenstein distance between the laws of solutions is Lipschitz continuous w.r.t. the initial values, which is stated in the following proposition.
	
	\begin{proposition}\label{lip1}
		There exists a constant $c>0$ such that
		\begin{align}
			W_2(P^f,P^g)\leq c\rho (f,g),\ f,g\in L^2_{tem}.
		\end{align}
	\end{proposition}
	\begin{proof}
		We denote by $u^f$ and $u^g$ the two solutions of SPDE \eqref{3.1} starting from $f,g\in L^2_{tem}$. For $\lambda>0$, we have
		\begin{align}\label{LL0}
			&\mathbb{E}~\sup_{t \leq T}\left\{\int_{\mathbb{R}}|u^f(t,x)-u^g(t,x)|^2e^{-2\lambda|x|}\mathrm{d}x\right\}\nonumber\\
			\leq & 3~\sup_{t \leq T}\left\{ \int_{\mathbb{R}}\left(\int_{\mathbb{R}} p_t(x, y)(f(y)-g(y)) \mathrm{d}y\right)^2 e^{-2 \lambda|x|} \mathrm{d}x\right\} \nonumber\\
			&+3~\mathbb{E}~\sup_{t\leq T}\left\{ \int_{\mathbb{R}}\left|\int_{0}^{t}\int_{\mathbb{R}} p_{t-s}(x,y)\left(b(u^f(s,y))-b(u^g(s,y))\right)\mathrm{d}y\mathrm{d}s\right|^2e^{-2\lambda |x|}\mathrm{d}x\right\} \nonumber\\
			&+3~\mathbb{E}\sup_{t\leq T}\left\{\int_{\mathbb{R}}
			\left|\int_{0}^{t}\int_{\mathbb{R}} p_{t-s}(x,y)\left(\sigma(u^f(s,y))-\sigma(u^g(s,y))\right)W(\mathrm{d}s,\mathrm{d}y)\right|^2e^{-2\lambda |x|}\mathrm{d}x\right\}.
		\end{align}
		By H\"{o}lder's inequality and \eqref{103.1}, we have
		\begin{align}\label{LL1}
			&\sup_{t \leq T}\left\{ \int_{\mathbb{R}}\left(\int_{\mathbb{R}} p_t(x, y)(f(y)-g(y)) \mathrm{d}y\right)^2 e^{-2 \lambda|x|} \mathrm{d}x\right\} \nonumber\\
			\leq&\sup_{t \leq T} \left\{\int_{\mathbb{R}} \left(\int_{\mathbb{R}} p_t(x, y)\mathrm{d}y\right)\left(\int_{\mathbb{R}} p_t(x, y)|f(y)-g(y)|^2 \mathrm{d}y\right)\,e^{-2 \lambda|x|} \mathrm{d}x\right\}\nonumber\\
			\leq& 2e^{2\lambda^2T}~ \int_{\mathbb{R}}|f(y)-g(y)|^2 e^{-2 \lambda|y|} \mathrm{d}y\nonumber\\
			=& 2e^{2\lambda^2T}\Vert f-g\Vert^2_{L^2_\lambda}.
		\end{align}
		Similarly,
		\begin{align}\label{LL2}
			&\mathbb{E}~\sup_{t\leq T}\left\{ \int_{\mathbb{R}}\left|\int_{0}^{t}\int_{\mathbb{R}} p_{t-s}(x,y)\left(b(u^f(s,y))-b(u^g(s,y))\right)\mathrm{d}y\mathrm{d}s\right|^2e^{-2\lambda |x|}\mathrm{d}x\right\} \nonumber\\
			\leq & 2e^{2\lambda^2T}TL^2_b ~ \int_{0}^{T} \mathbb{E}\sup_{r \leq s}\left(\int_{\mathbb{R}}|u^f(r,y)-u^g(r,y)|^2e^{-2\lambda|y|}\mathrm{d}y\right)\mathrm{d}s.
		\end{align}
		By Lemma \ref{moment2} and (\textbf{H2}) we have
		\begin{align}\label{LL3}
			&\mathbb{E}\sup_{t\leq T}\left\{\int_{\mathbb{R}}
			\left|\int_{0}^{t}\int_{\mathbb{R}} p_{t-s}(x,y)\left(\sigma(u^f(s,y))-\sigma(u^g(s,y))\right)W(\mathrm{d}s,\mathrm{d}y)\right|^2e^{-2\lambda |x|}\mathrm{d}x\right\}
			\nonumber\\
			\leq & \epsilon L^2_\sigma~ \mathbb{E}~\sup_{s \leq T} \int_{\mathbb{R}} |u^f(s,y)-u^g(s,y)|^2e^{-2\lambda |y|}\mathrm{d}y \nonumber\\
			&\quad+C_{\epsilon,\lambda,T} L^2_\sigma~\int_{0}^{T} \mathbb{E}~\sup_{r \leq s}\left(\int_{\mathbb{R}} |u^f(r,y)-u^g(r,y)|^2e^{-2\lambda|y|}\mathrm{d}y\right)\mathrm{d}s,
		\end{align}
		Then combining \eqref{LL1}-\eqref{LL3} with \eqref{LL0} together and using Gronwall's inequality, it follows that there exists some constant $c>0$ dependent on $\epsilon,\lambda,T,L_b,L_\sigma$ such that
		\begin{align*}
			\mathbb{E}~\sup_{t\leq T}\Vert u^f(t,\cdot)-u^g(t,\cdot)\Vert_{L^2_\lambda}\leq c\Vert f-g\Vert_{L^2_\lambda}.
		\end{align*}
		Note that the constant $c$ is increasing w.r.t. $\lambda$.  Hence there exists a constant $c$ independent of $n$ such that for all $n\geq 1$,
		\begin{align*}
			\mathbb{E}~ \sup_{t\leq T}\min\left\{1,\Vert u^f(t,\cdot)-u^g(t,\cdot)\Vert_{L^2_{1/n}}\right\}
			\leq c\min\left\{1,\Vert f-g\Vert_{L^2_{1/n}}\right\}.
		\end{align*}
		This implies that
		\begin{align}\label{lip}
			&W_2(P^f,P^g) \nonumber\\
			\leq &\mathbb{E}~\sup_{t\leq T} \rho(u^f(t,\cdot),u^g(t,\cdot))\nonumber\\
			= & \mathbb{E}~\sup_{t\leq T}\sum_{n=1}^{\infty} \frac{1}{2^n} \min\left\{1,\Vert u^f(t,\cdot)-u^g(t,\cdot)\Vert_{L^2_{1/n}}\right\}\nonumber\\
			\leq &\sum_{n=1}^{\infty}\frac{1}{2^n} \mathbb{E}~ \sup_{t\leq T}\min\left\{1,\Vert u^f(t,\cdot)-u^g(t,\cdot)\Vert_{L^2_{1/n}}\right\}\nonumber\\
			\leq & \sum_{n=1}^{\infty}\frac{1}{2^n} c\min\left\{1,\Vert f-g\Vert_{L^2_{1/n}}\right\}\nonumber\\
			=& c\rho(f,g).
		\end{align}
		The proof of this proposition is complete.
	\end{proof}
	With this proposition and Theorem \ref{Lresult}, we complete the proof of  Corollary \ref{Lrandom}.
	
	\vskip 0.5cm
	\noindent \emph{Proof of Corollary \ref{LCtemrandom}}
	
	From  Theorem \ref{Lresult} and Theorem \ref{Ctemresult}, we see that for $u_0\in L^2_{tem}\cap C_{tem}$,
	\begin{align}
		W_2(Q,P^{u_0})^2\leq CH(Q|P^{u_0}),\quad Q\in \mathcal{P}(C([0,T],L^2_{tem}\cap C_{tem})),
	\end{align}
	$P^{u_0}$ is the law of the solution of the SPDE \eqref{3.1} started at $u_0$.
	Also, according to Theorem 2.1 in \cite{WZ}, to prove Corollary \ref{LCtemrandom} we only need to establish the following proposition.
	\begin{remark}
		In fact, in $C_{tem}$ the Lipschitz property like Proposition \ref{lip1} does not exist, unless we consider the equation \eqref{3.1} in $C_{tem}\cap L^2_{tem}$.
	\end{remark}
	\begin{proposition}
		There exists a constant $\tilde{c}>0$ such that
		\begin{align}
			W_2(P^f,P^g)\leq \tilde{c}\rho (f,g),\ f,g\in L^2_{tem}\cap C_{tem}.
		\end{align}
	\end{proposition}
	\begin{proof}
		Let $u^f$($u^g$, respectively) be the unique solution of SPDE \eqref{3.1} with $u_0=f\in L^2_{tem}\cap C_{tem}$($g$, respectively).
		
		For $\lambda>0$, by \eqref{3.3} and Gronwall's inequality we can obtain
		\begin{align*}
		&\sup_{x\in\mathbb{R}} \left(\left|u^f(t,x)-u^g(t,x)\right|^2e^{-2\lambda |x|}\right) \\
		\leq &c_{\lambda,t,L_b} \left(\sup_{x\in\mathbb{R}}\left(|f(x)-g(x)|^2e^{-2\lambda |x|}\right) \right.\\
		&\left. +\sup_{(t,x)\in[0,T]\times\mathbb{R}}
		\left(\left|\int_{0}^{t}\int_{\mathbb{R}} p_{t-s}(x,y)\left(\sigma(u^f(s,y))-\sigma(u^g(s,y)\right)W(\mathrm{d}s,\mathrm{d}y)\right|^2e^{-2\lambda |x|}\right)
		\right),\ t\in[0,T].
		\end{align*}
		According to Lemma \ref{210121.1036}, we have for small $\epsilon>0$,
	\begin{align}\label{ZZ}
		&\mathbb{E} \sup_{(t,x)\in[0,T]\times\mathbb{R}}\left(\left|u^f(t,x)-u^g(t,x)\right|^2e^{-2\lambda |x|}\right) \nonumber\\
		\leq &c_{\lambda,T,L_b} ~\sup_{x\in\mathbb{R}}\left(|f(x)-g(x)|^2e^{-2\lambda |x|}\right) \nonumber\\
		&+\epsilon~c_{\lambda,T,L_b,L_\sigma} ~\mathbb{E}\sup_{(t,x)\in[0,T]\times\mathbb{R}}\left(\left|u^f(t,x)-u^g(t,x)\right|^2e^{-2\lambda |x|}\right) \nonumber\\
		&+c_{\epsilon,\lambda,T,L_b,L_\sigma} ~\mathbb{E}\int_{0}^{T}\int_{\mathbb{R}}\left|u^f(t,x)-u^g(t,x)\right|^2e^{-2\lambda|x|}\mathrm{d}x\mathrm{d}t.
	\end{align}
	And similarly to Lemma \ref{L2} we can prove
		\begin{align}\label{GG}
			\mathbb{E}\int_{\mathbb{R}}\left|u^f(t,x)-u^g(t,x)\right|^2e^{-2\lambda|x|}\mathrm{d}x\leq c_{\lambda,T,L_b,L_\sigma} \int_{\mathbb{R}} |f(y)-g(y)|^2e^{-2\lambda |y|}\mathrm{d}y.
		\end{align}
		Then by \eqref{ZZ} and \eqref{GG} and taking some appropriate $\epsilon$ we obtain
		\begin{align*}
			&\mathbb{E} \sup_{(t,x)\in[0,T]\times\mathbb{R}}\left(\left|u^f(t,x)-u^g(t,x)\right|^2e^{-2\lambda |x|}\right) \nonumber\\
			\leq &c'\left( \sup_{x\in\mathbb{R}}\left(|f(x)-g(x)|^2e^{-2\lambda |x|}\right)
			+\int_{\mathbb{R}} |f(y)-g(y)|^2e^{-2\lambda |y|}\mathrm{d}y\right).
		\end{align*}
	    for some constant $c'$ increasing w.r.t. $\lambda$.
		Hence,  there exists a constant $c''>0$, such that for all $\lambda\in(0, 1]$,
		\begin{align*}
			\mathbb{E}\sup_{t\in[0,T]}\varrho_\lambda(u^f(t,\cdot),u^g(t,\cdot))\leq c'' (\varrho_\lambda(f,g)+ \Vert f-g\Vert_{L^2_\lambda}).
		\end{align*}
		Consequently we have
		\begin{align*}
			\mathbb{E}~\sup_{t\in [0,T]}\varrho(u^f(t,\cdot),u^g(t,\cdot))
			= & \mathbb{E}~\sup_{t\in [0,T]}\sum_{n=1}^{\infty} \frac{1}{2^n} \min\left\{1,\varrho_{1/n} (u^f(t,\cdot),u^g(t,\cdot))\right\}\nonumber\\
			\leq &\sum_{n=1}^{\infty}\frac{1}{2^n} \mathbb{E}~ \sup_{t\in [0,T]}\min\left\{1,\varrho_{1/n} (u^f(t,\cdot),	u^g(t,\cdot)) \right\}\nonumber\\
			\leq & \sum_{n=1}^{\infty}\frac{1}{2^n} c''\min\left\{1,\varrho_{1/n}(f,g)+ \Vert f-g\Vert_{L^2_{1/n}}\right\}\nonumber\\
			\leq & c''(\varrho(f,g)+\rho(f,g)).
		\end{align*}
		Together with \eqref{lip}, we obtain
		\begin{align*}
			\mathbb{E}~\sup_{t\in[0,T]}\left[\varrho(u^f(t,\cdot),u^g(t,\cdot))+\rho (u^f(t,\cdot),u^g(t,\cdot))\right]\leq \tilde{c}(\varrho(f,g)+\rho(f,g)).
		\end{align*}
		The proof of this Proposition is complete.
	\end{proof}
	
	According to Theorem 2.1 in \cite{WZ}, we have completed the proof of Corollary \ref{LCtemrandom}.
	
	$\blacksquare$

\vskip 0.3cm
\noindent{\bf Acknowledgement}. 
This work is partially supported by the National Key R\&D Program of China (No. 2022YFA1006001), the National Natural Science Foundation of China (No. 12131019, No. 11721101, No. 12001516), the Fundamental Research Funds for the Central Universities (No. WK3470000031, No. WK3470000024, No. WK3470000016), and the School Start-up Fund (USTC) KY0010000036.

\end{document}